\documentclass[11pt,reqno,letterpaper]{amsart}

\usepackage{amssymb,amsmath,amsthm,amsfonts}
\usepackage{bbm,enumerate}
\usepackage{mathtools} 

\usepackage{bookmark}
\usepackage{hyperref}
\hypersetup{pdfstartview={FitH}}

\theoremstyle{plain}
\newtheorem{theorem}{Theorem}

\newtheorem{lemma}[theorem]{Lemma}

\theoremstyle{remark}
\newtheorem{remark}{Remark}

\numberwithin{equation}{section}

\newcommand{\eps}{\varepsilon}
\newcommand{\N}{\mathbb{N}}

\newcommand{\R}{\mathbb{R}}
\newcommand{\supp}{\operatorname{supp}}
\newcommand{\jap}[1]{\left\langle #1 \right\rangle}
\DeclarePairedDelimiter\abs{\lvert}{\rvert}
\DeclarePairedDelimiter\norm{\lVert}{\rVert}

\begin{document}

\title[Generic norm growth of powers of multipliers]{Generic norm growth of powers of homogeneous unimodular Fourier multipliers}

\author[A. Bulj]{Aleksandar Bulj}

\address{Department of Mathematics, Faculty of Science, University of Zagreb, Bijeni\v{c}ka cesta 30, 10000 Zagreb, Croatia}
\email{aleksandar.bulj@math.hr}



\begin{abstract}
For an integer $d\ge 2$, $t\in \mathbb{R}$ and a $0$-homogeneous function $\Phi\in C^{\infty}(\mathbb{R}^{d}\setminus\{0\},\mathbb{R})$, we consider the family of Fourier multiplier operators $T_{\Phi}^t$ associated with symbols $\xi\mapsto \exp(it\Phi(\xi))$ and prove that for a generic phase function $\Phi$, one has the estimate $\lVert T_{\Phi}^t\rVert_{L^p\to L^p} \gtrsim_{d,p, \Phi}\langle t\rangle ^{d\lvert\frac{1}{p}-\frac{1}{2}\rvert}$. That is the maximal possible order of growth in $t\to \pm \infty$, according to the previous work by V. Kova\v{c} and the author and the result shows that the two special examples of functions $\Phi$ that induce the maximal growth, given by V. Kova\v{c} and the author and independently by D. Stolyarov, to disprove a conjecture of Maz'ya actually exhibit the same general phenomenon.
\end{abstract}

\maketitle


\section{Introduction}
For $d\in \N$, $d\ge 2$ and $\Phi\in C^{\infty}(\R^{d}\setminus\{0\}, \R)$, a homogeneous function of degree 0, we consider a family of Fourier multiplier operators indexed by $t\in \R$, defined on the set of Schwartz functions $\mathcal{S}(\R^d)$ with
\begin{equation}
  \label{eq:T_def}
  T_{\Phi}^t f(x):=\int_{\R^d}e^{it\Phi(\xi)+i\xi \cdot x}\widehat{f}(\xi)d\xi.
\end{equation}
Since $\Phi$ is smooth and real-valued homogeneous function of degree 0, by \cite[\S III.3.5, Theorem~6]{Stein70book} we know that $T_{\Phi}^t$
has representation in the form
\[T_{\Phi}^tf=a_\Phi^t f + S_{\Phi}^tf, \quad f\in \mathcal{S}(\R^d),\]
where, denoting a spherical measure on $S^{d-1}$ with $\sigma_{d-1}$, $a_\Phi^t$ is a constant given by
\[a_\Phi^t:=\frac{1}{\sigma_{d-1}(S^{d-1})}\int_{S^{d-1}}e^{it\Phi(\xi)} d\sigma_{d-1}(\xi)\] 
and $S_{\Phi}^t$ is a singular integral operator defined on $\mathcal{S}(\R^d)$ as:
\[S_{\Phi}^tf(x):=\lim_{\eps\to 0}\int_{\abs{y}\ge \eps}\frac{\Omega_\Phi ^t(\frac{y}{\abs{y}})}{\abs{y}^d}f(x-y)dy\]
for some function $\Omega_\Phi^t\in C^{\infty}(S^{d-1},\R)$  that satisfies $\int_{S^{d-1}}\Omega_{\Phi}^t d\sigma_{d-1} =0$.

It is obvious from the definition that $\abs{a_\Phi^t}\leq 1$ and by the Calderon--Zygmund theory (see \cite[Theorem~1]{CZ56} or \cite[\S II.4.2, Theorem~3]{Stein70book}) we know that the singular integral operator $S_{\Phi}^t$ extends to an $L^p\to L^p$ bounded operator for all $p\in (1,\infty)$.

The question of the asymptotic behavior of $t\mapsto \norm{T_\Phi^t}_{L^p\to L^p}$ as $t\to\pm \infty$ was implicitly raised in \cite{MH76} and explicitly posed by V. Maz'ya in \cite[\S 4.2]{Mazya75problems}, where he asked whether there exists a finite constant $C_{d,p,\Phi}$ for which the following estimate holds for all homogeneous functions  $\Phi\in C^{\infty}(\R^d\setminus\{0\}, \R)$ of degree 0 and all $t\in \R$:
\[\norm*{T_{\Phi}^t}_{L^p\to L^p}\leq C_{d,p,\Phi} \jap{t}^{(d-1)\abs{\frac{1}{2}-\frac{1}{p}}}.\] 
The question was negatively resolved by V. Kova\v{c} and the author in \cite{BK23} for even $d$, by the following theorem. Since sharpness in $p\in (1,\infty)$ was a substantial part of \cite{BK23}, we will denote $p^{*}:=\max\{p,\frac{p}{p-1}\}$ and write the full statement of the theorem.
\begin{theorem}[\cite{BK23}, Theorem~2]
\label{thm:bk23}
    \begin{enumerate}[(a)]
        \item Fix an integer $d\ge 2$ and a homogeneous function $\Phi\in C^{\infty}(\R^{d}\setminus\{0\}, \R)$ of degree 0. There exists a finite constant $C_{d,\Phi}$ such that for every exponent $p\in (1,\infty)$ and $t\in\R$ we have
        \[\norm*{T^{t}_{\Phi}}_{L^p\to L^p}\leq C_{d,\Phi} (p^{*}-1)\jap{t}^{d\abs{\frac{1}{2}-\frac{1}{p}}}.\]

        \item Fix an even integer $d\ge 2$. There exists a homogeneous function $\Phi\in C^{\infty}(\R^{d}\setminus\{0\},\R)$ of degree 0 and a constant $c_{d,\Phi}>0$ such that for every exponent $p\in(1,\infty)$ and every nonzero integer $k$ we have
        \[\norm*{T^{k}_{\Phi}}_{L^p\to L^p}\geq c_{d,\Phi}(p^{*}-1) \abs{k}^{d\abs{\frac{1}{2}-\frac{1}{p}}}.\]
    \end{enumerate}
\end{theorem}
One has to note that the answer to the particular case $d=2$ of Maz'ya's question follows from the work of Dragi\v{c}evi\'{c}, Petermichl and Volberg \cite{DPV06}. This fact was overlooked in \cite{Mazya75problems}, when the question was formulated.

Furthermore, D. Stolyarov \cite{Sto22} independently, and using different techniques, proved the part (a) of Theorem \ref{thm:bk23} and showed the existence of a function $\Phi$ as in the part (b) for all $d\ge 2$, but without the sharp dependence in $p\in (1,\infty)$ in both parts. 

Both papers, \cite{BK23} and \cite{Sto22}, have short proofs for the upper bound. In \cite{BK23} authors used spherical harmonics to reduce the problem to $L^1\to L^{1,\infty}$ estimate for singular integrals with rough kernels of Seeger and Tao, while in \cite{Sto22} the author reduced the problem to the sharp version of the H\"{o}rmander--Mikhlin multiplier theorem.

However, in \cite{BK23} we gave an extremely specific class of functions $\Phi$ for which the bound is asymptotically sharp for $t\to\pm \infty$ and in \cite{Sto22} the author gave one particular function for which the bound is asymptotically sharp. Both papers have relatively long proofs from which the cause of the exact worst asymptotics was not apparent. In this paper we give a short proof of existence of a general phenomenon that drives the growth of powers of norms for a generic symbol, proving that all, when $d=2$, and ``almost all", when $d\ge 3$,  unimodular homogeneous symbols of degree 0 are counterexamples to the asymptotic order of growth conjectured by V. Maz'ya in \cite[\S 4.2]{Mazya75problems} and, in fact, exhibit the worst possible asymptotics. More precisely, we prove the following theorem --- for the definition of Japanese bracket $\jap{\cdot}$ see \ref{subs:notation} below and for the definition of Whitney topology see \ref{subs:preliminaries} below. 

\begin{theorem}
\label{thm:intro}
Let $d\in \N$, $d\ge 2$, $p\in (1,\infty)$ and $t\in \R$. 
\begin{enumerate}[(a)]
    \item 
    \label{thm:intro_a}
    For $d=2$ and any nonconstant homogeneous function $\Phi\in C^{\infty}(\R^2\setminus\{0\},\R)$ of degree 0 there exists a constant $c_{d,p,\Phi}>0$ such that
    \[\norm*{T^{t}_{\Phi}}_{L^p\to L^p}\ge c_{d,p,\Phi} \jap{t}^{2\abs{\frac{1}{2}-\frac{1}{p}}}.\]
    \item 
    \label{thm:intro_b}
    For $d\ge 3$, there is a dense open set $\mathcal{G}$ in the Whitney topology on $C^{\infty}(S^{d-1},\R)$ such that for all $\phi\in\mathcal{G}$ there exists a constant $c_{d,p,\phi}>0$ for which the 0-homogeneous extension $\Phi\in C^{\infty}(\R^d\setminus\{0\},\R)$ of $\phi$, defined as $\Phi(\xi):=\phi(\frac{\xi}{\abs{\xi}})$, satisfies
    \[\norm*{T^{t}_{\Phi}}_{L^p\to L^p}\ge c_{d,p,\phi} \jap{t}^{d\abs{\frac{1}{2}-\frac{1}{p}}}.\]
\end{enumerate}
\end{theorem}
In topology, the property that holds on a dense open set (or, more generally, on the complement of a countable union of nowhere dense sets) is called generic. Therefore, since a 0-homogeneous function on $\R^{d}\setminus\{0\}$ is uniquely defined by its restriction on the sphere $S^{d-1}$, the previous theorem says that multipliers associated with powers of a generic 0-homogeneous unimodular symbol exhibit the asymptotically maximal possible order of growth of $L^p\to L^p$ norms.

The part \eqref{thm:intro_b} of \ref{thm:intro} will be a consequence of the following theorem, which can be of its own interest, so we state it here. For the definition of nondegeneracy of a critical point see \ref{subs:preliminaries} below.

\begin{theorem}
\label{thm:main}
    Let $d\in \N$, $d\ge 2$, $p\in (1,\infty)$ and $t\in \R$. For a homogeneous function $\Phi \in C^{\infty}(\R^d\setminus\{0\},\R)$ of degree 0 such that $\Phi\vert_{S^{d-1}}$ has a nondegenerate local minimum or maximum, there exists a constant $c_{d,p,\Phi}>0$ such that
    \[\norm{T^t_{\Phi}}_{L^p(\R^d)\to L^p(\R^d)} \geq c_{d,p,\Phi}\jap{t}^{d\abs{\frac{1}{2}-\frac{1}{p}}}.\]
\end{theorem}

Choosing $\Phi(\xi)= \frac{\xi_1}{\abs{\xi}}$ in Theorem \ref{thm:main}, one gets the asymptotics for the so-called Riesz group (a.k.a.\@ the Poincar\'{e}--Riesz--Sobolev group) that appears, when $d=3$, in the analysis of the Navier--Stokes equations in a rotating frame; see \cite[\S2]{GIMM06}, \cite[\S4]{GIMM07}, \cite[Eq.~(1.3)]{GS13}, or \cite[Eq.~(23)]{FM12}, while the asymptotics of the same symbol when $d=2$ was studied in \cite{DPV06}. 
In fact, the lower bound for $d\geq3$ established by Stolyarov \cite{Sto22} is equivalent to the particular case of Theorem \ref{thm:main} for this particular choice of $\Phi$.

Finally, for the sake of completeness, we mention that the asymptotics of $t\mapsto \norm{T_{\Phi}^t}_{L^p\to L^p}$ is an uninteresting problem when $d=1$ because for any $\Phi$ as before and $t\in \R$, one can write $T_{\Phi}^t$ as a bounded linear combination of the identity and the Hilbert transform to get the bound $\norm{T_{\Phi}^t}_{L^p\to L^p}\lesssim 1$ for all $t\in \R$.

\subsection{Notation.}
\label{subs:notation}
For $x\in\R^d$, $x=(x_1,\dots, x_d)$ we use $\abs{x}$ to denote standard Euclidean norm, we use $\jap{x}:=(1+\abs{x}^2)^{\frac{1}{2}}$ to denote Japanese bracket and we use the notation $x_{-}:=(x_1,\dots, x_{d-1})$. We will sometimes write ``$-$" in subscript to emphasize that we are working with the first $d-1$ coordinates only. 
When working with matrices, we identify vectors in $\R^d$ with matrices of size $d\times 1$. For a matrix $A\in M_d(\R)$, we write $A>0$ to denote that it is positive definite. For a function $f:\R^d\to\R$ we use notation $\nabla f$ to denote the $1\times d$ matrix $\nabla_x f :=\left[\frac{\partial f}{\partial x_j}\right]_j$ and $Hf$ to denote the $d\times d$ Hessian matrix $H_x f := \left[\frac{\partial^2f}{\partial x_j \partial x_k}\right]_{j,k}$. We will suppress $x$ in subscript if the variables upon which we differentiate are clear from the context.
We write $B_{d}(c,r)$ to denote the ball in $\R^d$ with center $c$ and radius $r$ and use notation $S^{d-1}$ to denote the unit sphere around 0 in $\R^d$. 
If $A,B$ are some expressions, then for some (possibly empty) set of indices, $P$, the expression $A\lesssim_P B$ means that there is a constant $C>0$ depending on $P$ such that $A\leq CB$. The expression $A\sim_P B$ means that $A\lesssim_P B$ and $B\lesssim_P A$. 
Finally, since we care the most about the phase of the oscillatory integral, to suppress writing $2\pi$ in the exponent, we use the following normalization of the Fourier transform of a function $f\in L^1(\R^d)$:
\[\widehat{f}(\xi):= \frac{1}{(2\pi)^d}\int_{\R^d}f(x)e^{-i\xi \cdot x}dx.\]

\section{Idea of the proof and preliminaries}
\label{sec:idea}
Statements of the theorems are interesting only when $\abs{t}$ is large, so the idea of the proof is to follow the approach from \cite[Exercise 2.34]{TaoBook}, used to study the asymptotics of $L^p$ behavior of the Schr\"{o}dinger propagator (defined by \eqref{eq:T_def} with $\Phi(\xi)=\abs{\xi}^2$), that relies on the following observation. 
When $p\in (1,2]$, and $t>0$, the correct asymptotics can be written as $t^{-\frac{d}{2}}\times t^{\frac{d}{p}}$, what can be interpreted as base $\times$ height approximation of the $L^p$ norm of a function that resembles a bump of height $t^{-\frac{d}{2}}$ on the ball of radius $t$ with a rapidly decreasing tail, both in $x$ and $t$, outside of it.
In the case of the Schr\"{o}dinger propagator, Young's inequality for convolutions with an explicit calculation of the kernel, gives $\norm{T_{\abs{\cdot}^2}^tf}_{L^{\infty}}\lesssim t^{-\frac{d}{2}}$ and the method of nonstationary phase applied to \eqref{eq:T_def} for $x$ outside of the ball of radius $\gtrsim t$ proves that the function has a rapidly decreasing tail, both in $x$ and $t$.
The two estimates imply $\norm{T_{\abs{\cdot}^2}^tf}_{L^p}\lesssim t^{\frac{d}{p}-\frac{d}{2}}$ as $t\to \infty$ for any $p\in (1,\infty)$ and then log-convexity of $L^p$ norms applied to $p$ and $p' = \frac{p}{p-1}$, together with the fact that $T_{\abs{\cdot}^2}$ is a unitary operator on $L^2$, transfers the upper bounds to lower bounds.

On the contrary, in the case of the homogeneous multipliers of degree 0, the kernel is singular, so one cannot use Young's convolution inequality to control the $L^{\infty}$ size of $T_{\Phi}^tf$ and needs a different approach. 
The obvious method to try is the method of (non-)stationary phase, but since $\Phi$ is homogeneous of degree 0, it follows that for all $\xi\in \R^d\setminus\{0\}$
\[0=\frac{d^2}{d h^2}(\Phi(\xi+h \xi))\Big\vert_{h=0} = \xi^{\top }H \Phi(\xi)\xi, \]
implying, together with the fact that the Hessian of the function $\xi\mapsto \langle \xi,x\rangle$ is equal to $0$ for all $x,\xi \in\R^d$, that all stationary points of the phase are degenerate, so they don't fall under the scope of the classical method of stationary phase. 

We are able to circumvent this problem and reduce the problem to the classical method of stationary phase by transforming the integral representation in the case of an appropriately localized function $\widehat{f}$ using the change of variables in the integral (see \eqref{eq:T_transf} below). The reason why the method works better when $d=2$ is the fact that regularity of the Hessian of the phase in the transformed expression does not depend on the second derivative, contrary to the $d\ge 3$ case.

Applying the method of stationary phase to the modified representation \eqref{eq:T_transf} with a localized function $\widehat{f}$, it turns out again that the best $L^{\infty}$ bound for $T_{\Phi}^tf$ is bigger than $t^{-\frac{d}{2}}$, so the function $T_{\Phi}^t f$ does not resemble a bump function of height $t^{-\frac{d}{2}}$ as it did in the case of the Schr\"{o}dinger propagator.

However, using implicit and inverse function theorems we are able to show the existence of the set of $x$'s of measure $\sim t^{d}$ on which the modified phase in \eqref{eq:T_transf} is stationary and nondegenerate, so the method of stationary phase gives the required asymptotics $\abs{T^t_{\Phi}f(x)} \sim_{d,p,\Phi} t^{-\frac{d}{2}}$ on the given set. 
Using the base $\times$ height bound, this implies the required asymptotics, that is $\norm{T^t_{\Phi}f}_p \gtrsim_{d,p,\Phi} t^{\frac{d}{p} - \frac{d}{2}}$.

In the remainder of the section we quickly recall the main theorems from the Morse theory and the method of stationary phase that we will be used in the proof.

\subsection{Preliminaries}
\label{subs:preliminaries}
An introduction to differential topology and the Morse theory can be found, for example, in \cite{Hirsch76book}. We recall the basic facts needed in this paper. For a manifold $M$, we say that a function $f:M\to \R$ has nondegenerate critical point at $p\in M$ if $\nabla (f\circ \psi^{-1}) (\psi(p))=0$ and $(H(f\circ \psi^{-1}))(\psi(p))$ is a regular matrix for some local chart $\psi$ at $p$. Direct calculation shows that the definition is independent of a chosen chart around a critical point. 
A function $f\in C^2(S^{d-1},\R)$ is called a Morse function if all critical points of $f$ are nondegenerate. 
The definition of Whitney topology on $C^{\infty}(S^{d-1},\R)$ can be found in \cite[\S II.1]{Hirsch76book}.
We will use a known theorem that the set of Morse functions in $C^{\infty}(S^{d-1},\R)$ is a dense open set in the Whitney topology on $C^{\infty}(S^{d-1}, \R)$. For the reference see \cite[\S VI, Theorem~1.2]{Hirsch76book}.

One can find a thorough introduction to oscillatory integrals in \cite[\S 8]{Stein93book}. Definition of nondegeneracy of a critical point in the case $M=\R^d$ reduces to the statement that $f:\R^d\to \R$ has a nondegenerate critical point at $p\in \R^d$ if $\nabla f(p)=0$ and $\det Hf(p)\neq 0$.

We will need the following theorem that follows by combining the method of stationary phase \cite[\S 8, Proposition~6]{Stein93book} around a critical point with the method of nonstationary phase \cite[\S 8, Proposition~4]{Stein93book} away from it. 
\begin{theorem}
    \label{thm:stat_phase}
    Let $\psi\in C^{\infty}_c(\R^d,\R)$ and let $\Phi\in C^2(\R^d,\R)$ be a function that has a unique critical point on the support of the function $\psi$, call it $\xi_0$. If $\xi_0$ is a nondegenerate critical point of $\Phi$, the following holds
    \[\int_{\R^d}e^{it \Phi(\xi)}\psi(\xi) d\xi = Ce^{it\Phi(\xi_0)}t^{-\frac{d}{2}} + O(t^{-\frac{d}{2}-1}),\quad t\to\infty,\]
    where $C =\psi(\xi_0)(2\pi)^{\frac{d}{2}}e^{\frac{i\pi}{4}\operatorname{sgn}(H\Phi(\xi_0))}\abs{\det H\Phi(\xi_0)}^{-\frac{1}{2}}$ and $\operatorname{sgn}(H\Phi(\xi_0))$ denotes the number of positive eigenvalues minus the number of negative eigenvalues of the matrix $H\Phi(\xi_0)$.
\end{theorem}

\section{Proofs}
Theorem \ref{thm:intro} will essentially follow from the Theorem \ref{thm:main}, but before we proceed to the proof of the Theorem \ref{thm:main}, we prove the following technical lemma that is crucial for the proof of both theorems. It proves the existence of large set of $x$'s for which the modified phase in \eqref{eq:T_transf} below is stationary and nondegenerate.

\begin{lemma}
\label{lem:main}
    Let $d\in \N$, $d\ge 2$ and $\phi\in C^2(\R^{d-1},\R)$. Suppose that either $H\phi (0) > 0$ or $d=2$ and $\phi'(0)\neq 0$.
    For $x,\xi\in \R^d$ define
    \begin{equation}
        \label{eq:phi_def}
        \Phi_x(\xi):=\phi(\xi_{-}) + \xi_d(\langle \xi_{-},x_{-} \rangle +x_d).
    \end{equation}
    There exist an open set $U\in \R^d$ and an open set $V\subset \R^{d-1}\times (\frac{1}{4},4)$ for which there is a unique function $g:U\to V$ such that for all $x\in U$ it holds $\nabla_{\xi} \Phi_x(g(x)) = 0$ and the matrix $H_{\xi}\Phi_x(g(x))$ is regular.
\end{lemma}
\begin{proof}
    Suppose first that $H\phi(0)>0$.
    Define $F:\R^{2d}\to \R^d$ with:
    \[F(\xi,x)= \nabla_\xi \Phi_x(\xi) = \begin{bmatrix}
        \nabla\phi(\xi_{-}) +\xi_dx_{-}^{\top} & \xi_1x_1+\dots +\xi_{d-1}x_{d-1}+x_d
    \end{bmatrix}.\]
    We want to apply the implicit function theorem to prove the existence of function $g$ such that $F(g(x),x)=0$. First observe that
    \[ \nabla_\xi F(\xi,x) = 
    \begin{bmatrix}
        H\phi(\xi_{-}) & x_{-}\\
        x_{-}^{\top} & 0
    \end{bmatrix}.\]
    Applying determinant to both sides of the block-matrix identity
    \[
    \begin{bmatrix}
        I_{d-1} & 0\\
        -x_{-}^{\top}(H\phi(\xi_{-}))^{-1} & 1
    \end{bmatrix}
    \begin{bmatrix}
        H\phi(\xi_{-}) & x_{-}\\
        x_{-}^{\top} & 0
    \end{bmatrix}=
    \begin{bmatrix}
        H\phi(\xi_{-}) & x_{-}\\
        0 & -x_{-}^{\top}(H\phi(\xi_{-}))^{-1}x_{-}
    \end{bmatrix}
    \]
    it follows that 
    \[\det \nabla_\xi F(\xi,x) = -\langle (H\phi(\xi_{-}))^{-1}x_{-},x_{-}\rangle \det H\phi(\xi_{-}).\]
    Since $H\phi(0) > 0$, there exists $\eps>0$ such that $ H\phi(\xi_{-})>0$ for all $\xi_{-}\in B_{d-1}(0,\eps)$. Furthermore, since the inverse of a positive definite matrix is positive definite, we conclude that $\nabla_\xi F(\xi,x)$ is regular whenever $\xi_{-}\in B_{d-1}(0,\eps)$ and $x_{-}\neq 0$.

    In order to apply the implicit function theorem, we need to show that there exist $\xi^0\in B_d(0,\eps)\times (\frac{1}{4},4)$ and $x^0\in \R^d$ with $x^0_{-}\neq 0$ such that $F(\xi^0,x^0)=0$.
    
    Since $H\phi(0)$ is regular, because of the inverse function theorem there exist an open set $A\subset B_{d-1}(0,\eps)$ and an open set $B\subset \R^{d-1}$, such that $\nabla \phi : A\to B$ is bijective. Taking $x^0_{-}\in (-B)\setminus\{0\}$ and $\xi^0_{-}$ such that $\nabla\phi(\xi^0_{-})=-x^0_{-}$ and defining: 
    \[\xi^0 = (\xi^0_{-},1), \quad x^0 = \left( x^0_{-}, -\langle \xi_{-}^{0},x^{0}_{-}\rangle \right), \] 
    we can see that $F(\xi^0,x^0)=0$. 
    
    Therefore, the implicit function theorem implies the existence of open sets $U'\ni x^0$ and $V\subset B(0,\eps)\times (\frac{1}{4},4)$ (the second inclusion follows from the fact that $\xi^0_d=1$ by shrinking $U'$ if necessary) and a unique function $g:U'\to V$ such that $F(g(x),x)=0$ for every $x\in U'$. If we choose $U$ to be a subset of $U'$ such that all $x\in U$ satisfy $x_{-}\neq 0$, the regularity of $H_\xi\Phi_x$ follows from the fact that $H_{\xi}\Phi_y(\xi) = \nabla_{\xi} F(\xi ,x)$ and the previous conclusion of regularity of $\nabla_{\xi} F(\xi ,x)$.

    When $d=2$ and $\phi'(0)\neq 0$, solving
    \[F(\xi, x)= \begin{bmatrix}
        \phi'(\xi_1)+\xi_2x_1 & \xi_1x_1+x_2
    \end{bmatrix} = 0,\]
    one can see that the unique function $g:\R^2\setminus(\{0\}\times \R)\to \R^2$ for which $F(g(x),x)=0$ is given by 
    \[g(x_1,x_2) = \left(-\frac{x_2}{x_1}, -\frac{\phi'(-\frac{x_2}{x_1})}{x_1}\right).\] 
    Also, observe that the matrix 
    \[\nabla_{\xi}F(\xi,x) = 
    \begin{bmatrix}
     \phi''(\xi_1) & x_1\\
     x_1 & 0
    \end{bmatrix}\]
    is regular whenever $x_1\neq 0$, regardless of $\phi''(\xi_1)$. 
    However, to satisfy the assumption that $V\subset \R^{d-1}\times (\frac{1}{4},4)$, one has to restrict $x$ to a smaller set. 
    Without loss of generality, we may assume that $c:=\phi'(0) > 0$. By continuity, there exists $\delta>0$ such that $\phi'(\xi_1)\in (\frac{c}{2},2c)$ for all $\xi_1\in (-\delta, \delta)$. 
    Therefore, if $x_1\in (-2c, -\frac{c}{2})$ and $x_2\in (-\frac{c\delta}{2}, \frac{c\delta}{2})$, then one has $-\frac{x_2}{x_1}\in (-\delta, \delta)$ and
    $\frac{-\phi'(-\frac{x_2}{x_1})}{x_1}\in (\frac{1}{4},4)$, giving the proof of the theorem with $U=(-2c, -\frac{c}{2})\times (-\frac{c\delta}{2}, \frac{c\delta}{2})$ and $V=\R\times (\frac{1}{4},4)$.
\end{proof}

\begin{proof}[Proof of Theorem \ref{thm:main}]
Using duality of $L^p$ spaces and the fact that 
\[\langle T_{\Phi}^t u,v \rangle  = \langle u, T^{-t}_{\Phi}v\rangle = \langle T_{\Phi}^t \Tilde{v},\Tilde{u} \rangle,\]
where $\Tilde{u}(x):=\overline{u(-x)}$, we can, without loss of generality, assume that $t\ge 0$ and $p\in (1,2]$.

By composing the function $\Phi$ with the rotation, if necessary, we can assume that the function $\Phi\vert_{S^{d-1}}$ has a local minimum at $e_d:=(0,\dots, 0,1)\in\R^d$. From the fact that it is nondegenerate, we know that the function $\phi:\R^{d-1}\to \R$ defined as restriction of $\Phi$ to the hyperplane $\langle \xi,e_d \rangle = 1$: 
\[\phi(\xi_1,\dots, \xi_{d-1}):=\Phi(\xi_1,\dots, \xi_{d-1},1)\]
satisfies $H\phi(0) >0$. Indeed, observing that
\[\phi(\xi_{-})= \Phi\left( \xi_{-}, 1 \right) =\Phi\vert_{S^{d-1}}\left( \frac{\xi_{-}}{\sqrt{1+\abs{\xi_{-}}^2}}, \frac{1}{\sqrt{1+\abs{\xi_{-}}^2}} \right),\]
the statement follows by the definition of nondegeneracy of $\Phi\vert_{S^{d-1}}$ and the fact that the function $\xi_{-}\mapsto \frac{1}{\sqrt{1+\abs{\xi_{-}}^2}}(\xi_{-},1)$ is the inverse of the chart of $S^{d-1}$ at $e_d$ given as $\xi\mapsto \frac{1}{\xi_d}\xi_{-}$.  

To reduce the integral to the correct form for application of the method of stationary phase, observe that
\[T_{\Phi}^tf(tx) =\int_{\R^d}e^{it(\Phi(\xi)+i\xi\cdot x)}\widehat{f}(\xi)d\xi.\]

Furthermore, observe that for any $\xi\in \R^{d-1}\times(0,\infty)$ one has $\Phi(\xi) = \phi(\frac{1}{\xi_d}\xi_{-})$. Since the function $\Lambda(\xi):=\xi_d(\xi_{-},1)$ is a $C^{\infty}$ diffeomorphism from $\R^{d-1}\times \R_{+}$ onto itself,  for any $f$ such that $\operatorname{supp}\widehat{f} \subset \R^{d-1}\times  \R_{+}$ the change of variables $\xi = \Lambda(\xi')$ gives
\begin{equation}
    \label{eq:T_transf}
    T_{\Phi}^tf(tx) = \int_{\R^d} e^{it(\phi(\xi_{-}) + \xi_d(\langle \xi_{-},x_{-} \rangle +x_d))}\widehat{f}(\xi_d\xi_{-},\xi_d)\xi_d^{d-1}d\xi.
\end{equation}

Let $U$,$V$ and $g$ be as in Lemma \ref{lem:main}. Since $\Lambda$ a $C^{\infty}$ diffeomorphism on $V$ and $V$ is open, there exist $\xi_0\in \Lambda(g(U))$ and a ball $B_d(\xi_0,\eps)\subset \Lambda(V)$. 
We choose a function $f\in \mathcal{S}(\R^d)$ such that $\supp\widehat{f}\subset B_d(\xi_0,\eps)$ and $\widehat{f}(\xi)=1$ for $\xi \in B_d(\xi_0,\frac{\eps}{2})$. Denoting $F(\xi):=\widehat{f}(\Lambda(\xi))\xi_d^{d-1}$ and  
\[U_1=\{x\in U: \Lambda(g(x))\in B_d(\xi_0,\frac{\eps}{2})\},\] 
the continuity of $\Lambda \circ g$ and the fact that $\xi_d\sim 1$ on $V$ imply that $U_1$ is an open set such that $\abs{F\circ g\vert_{U_1}}\gtrsim 1$. 
Denoting $\Phi_x(\xi)$ as in \eqref{eq:phi_def}, the existence and uniqueness of the function $g$ in Lemma \ref{lem:main} imply that for any $x\in U$, the function $\Phi_x$ has a unique stationary point in  $V\supset \supp F$. 
Therefore, from \eqref{eq:T_transf}, Theorem \ref{thm:stat_phase} and the lower bound for $F\circ g$ on $U_1$, we have
\begin{equation}
\label{eq:lower}
    \begin{split}
            \int_{\R^d} \abs{T_{\Phi}^tf(tx)}^p dx 
        & \ge \int_{U}\abs{T_{\Phi}^tf(tx)}^p dx\\
        & = \int_{U} \abs*{\int_{\R^d} e^{it\Phi_x(\xi)}F(\xi)d\xi}^p dx\\
        &= \int_{U} \abs*{t^{-\frac{d}{2}}(2\pi)^{\frac{d}{2}}F(g(x))\abs{\det H\Phi(g(x))}^{-\frac{1}{2}} + O_x(t^{-\frac{d}{2}-1})}^pdx\\
        &\gtrsim_{d,p} t^{-\frac{dp}{2}}\int_{U_1}\abs*{\abs{\det H\Phi(g(x))}^{-\frac{1}{2}} + O_x(t^{-1})}^p dx,
    \end{split}
\end{equation}
Taking the $p$-th root and applying Fatou's lemma, we have:
\begin{align*}
    \liminf_{t\to\infty} \frac{\norm{T_{\Phi}^tf}_{L^p}}{t^{\frac{d}{p}-\frac{d}{2}}} 
    &= \liminf_{t\to\infty} \frac{\norm{T_{\Phi}^tf(t\cdot)}_{L^p}}{t^{-\frac{d}{2}}} \\
    &\gtrsim_{d,p} \liminf_{t\to\infty} \left(\int_{U_1}\abs*{\abs{\det H\Phi(g(x))}^{-\frac{1}{2}} + O_x(t^{-1})}^p dx \right)^{\frac{1}{p}} \\
    &\gtrsim_{d,p} \left(\int_{U_1}\abs*{\abs{\det H\Phi(g(x))}^{-\frac{1}{2}}}^p dx \right)^{\frac{1}{p}}\\
    &\gtrsim_{d,p,\Phi} 1.
\end{align*}
Since $f$ was fixed, one has $\norm{f}_{L^p}\sim_p 1$, so the calculation implies that
\[\liminf_{t\to\infty}\frac{\norm{T_{\Phi}^t}_{L^p\to L^p}}{t^{\frac{d}{p}-\frac{d}{2}}}\ge \liminf_{t\to\infty}\frac{\norm{T_{\Phi}^tf}_{L^p}}{\norm{f}_{L^p} t^{\frac{d}{p}-\frac{d}{2}}}\gtrsim_{d,p,\Phi} 1,\]
giving the proof of the theorem for $t\ge M$, where $M\in \R_{+}$ is an absolute constant. 

The case $t\in [0,M]$, can be proved using soft methods. Fix any nonzero function $f\in \mathcal{S}(\R^d)$. The fact that the operator $T_{\Phi}^t$ is a unitary operator on $L^2$ for any $t\in \R$ implies that $T_{\Phi}^tf$ is not a zero function for any $t\in \R$. Therefore, for all $t\in \R$ one has $\norm{T_{\Phi}^tf}_{L^p}>0$. 
Furthermore, for any fixed $x\in \R^d$ and $t_0\in\R$, from \eqref{eq:T_def} and Lebesgue's dominated convergence theorem we have 
\[\lim_{t\to t_0} T_{\Phi}^tf(x) = T_{\Phi}^{t_0}f(x).\]
Fatou's lemma then implies that
\[\lim_{t\to t_0} \norm{T_{\Phi}^tf}_{L^p} \ge \norm{T_{\Phi}^{t_0}f}_{L^p},\]
meaning that the function $t\mapsto \norm{T_{\Phi}^tf}_{L^p}$ is lower semicontinuous. Lower semicontinuous function attains a minimum on the compact interval $[0,M]$ and it must be positive by the previous observation. Finally, using the fact that $\jap{t}^{\frac{d}{p}-\frac{d}{2}} \sim_{M} 1$ for $t\in [0,M]$, one gets the required lower bound in the range $[0,M]$, giving the proof of the theorem. 
\end{proof}

Finally, we prove Theorem \ref{thm:intro}.
\begin{proof}[Proof of Theorem \ref{thm:intro}] 
    Let us prove the part \eqref{thm:intro_a}.
    Since $\Phi\vert_{S^{d-1}}$ is not constant, there exist a point $\xi_0\in S^{d-1}$ and a chart $\psi$ at $\xi_0$ for which $\nabla (\Phi\vert_{S^{d-1}}\circ \psi^{-1}) (\psi(\xi_0))\neq 0$.
    By composing the function $\Phi$ with rotation, if necessary, we can assume that $\xi_0= e_2$ implying that $\partial_1\Phi(e_2)\neq 0$. Defining $\phi$ as in the proof of Theorem \ref{thm:main}, Lemma \ref{lem:main} gives the same conclusion needed to repeat the the proof of the Theorem \ref{thm:main} verbatim, thus proving the part \eqref{thm:intro_a}.

    We continue to the proof of part \eqref{thm:intro_b}.
    Let $\phi:S^{d-1}\to\R$ be any Morse function on the sphere. Since sphere is a compact set, it has a minimum and the fact that function is Morse implies that the minimum is nondegenerate. Applying the Theorem \ref{thm:main} to the function $\Phi(\xi):=\phi(\frac{\xi}{\abs{\xi}})$, we can see that all Morse functions satisfy the required asymptotics. From \cite[\S VI, Theorem~1.2]{Hirsch76book} we know that the set of Morse functions is an open dense set in the standard topology on $C^{\infty}(S^{d-1},\R)$, so the statement follows.
\end{proof}

\section{Closing remarks}

\begin{remark}
    The following question remains open and it could be interesting.

    When $d\ge 3$, does there exist a nonconstant unimodular homogeneous function $\Phi\in C^{\infty}(\R^d\setminus\{0\},\R)$ of degree 0 such that
    \[\liminf_{t\to\infty} \frac{\norm{T_\Phi^t}_{L^p\to L^p}}{t^{d\abs{\frac{1}{2}-\frac{1}{p}}}}=0 ?\]

\end{remark}

\begin{remark}
    When $d=2$, one can modify the approach from the proof of \cite[Theorem~3]{BK23} to prove the sharp estimate both in $p$ and $t$ for all nonconstant 0-homogeneous unimodular functions $\Phi\in C^{\infty}(S^1,\R)$, but it is not clear how to extend that approach to $d \ge 3$ nor it is apparent how to modify the current approach to imply the sharpness of the estimate both in $p$ and $t$.
\end{remark}


\section*{Acknowledgements}
I would like to thank my advisor Vjekoslav Kova\v{c} for numerous insightful discussions and meticulous reading of the paper. I would also like to thank Matko Grbac for suggesting a simple block matrix identity in Lemma \ref{lem:main} that made the presentation cleaner.


\bibliography{lower_bounds}{}

\begin{thebibliography}{10}

\bibitem{BK23}
Aleksandar Bulj and Vjekoslav Kova\v{c}.
\newblock Asymptotic behavior of {$L^p$} estimates for a class of multipliers
  with homogeneous unimodular symbols.
\newblock {\em Trans. Amer. Math. Soc.}, 376(7):4539--4567, 2023.

\bibitem{CZ56}
Alberto~Pedro Calder\'{o}n and Antoni Zygmund.
\newblock On singular integrals.
\newblock {\em Amer. J. Math.}, 78:289--309, 1956.

\bibitem{DPV06}
Oliver Dragi\v{c}evi\'{c}, Stefanie Petermichl, and Alexander Volberg.
\newblock A rotation method which gives linear {$L^p$} estimates for powers of
  the {A}hlfors-{B}eurling operator.
\newblock {\em J. Math. Pures Appl. (9)}, 86(6):492--509, 2006.

\bibitem{FM12}
Franco Flandoli and Alex Mahalov.
\newblock Stochastic three-dimensional rotating {N}avier-{S}tokes equations:
  averaging, convergence and regularity.
\newblock {\em Arch. Ration. Mech. Anal.}, 205(1):195--237, 2012.

\bibitem{GIMM06}
Yoshikazu Giga, Katsuya Inui, Alex Mahalov, and Shin'ya Matsui.
\newblock Navier-{S}tokes equations in a rotating frame in {$\mathbb{R}^3$}
  with initial data nondecreasing at infinity.
\newblock {\em Hokkaido Math. J.}, 35(2):321--364, 2006.

\bibitem{GIMM07}
Yoshikazu Giga, Katsuya Inui, Alex Mahalov, and Shin'ya Matsui.
\newblock Uniform local solvability for the {N}avier-{S}tokes equations with
  the {C}oriolis force.
\newblock In {\em Kyoto {C}onference on the {N}avier-{S}tokes {E}quations and
  their {A}pplications}, RIMS K\^{o}ky\^{u}roku Bessatsu, B1, pages 187--198.
  Res. Inst. Math. Sci. (RIMS), Kyoto, 2007.

\bibitem{GS13}
Yoshikazu Giga and J\"{u}rgen Saal.
\newblock An approach to rotating boundary layers based on vector {R}adon
  measures.
\newblock {\em J. Math. Fluid Mech.}, 15(1):89--127, 2013.

\bibitem{Hirsch76book}
Morris~W. Hirsch.
\newblock {\em Differential topology}, volume No. 33.
\newblock Springer-Verlag, New York-Heidelberg, 1976.

\bibitem{Mazya75problems}
Vladimir~Gilelevich Maz'ya.
\newblock Seventy five (thousand) unsolved problems in analysis and partial
  differential equations.
\newblock {\em Integral Equations Operator Theory}, 90(2):Paper No. 25, 44,
  2018.

\bibitem{MH76}
Vladimir~Gilelevich Maz'ya and Ju.~E. Ha\u{\i}kin.
\newblock The continuity of singular integral operators in normed spaces.
\newblock {\em Vestnik Leningrad. Univ.}, 1(Mat. Meh. Astronom. vyp. 1):28--34,
  160, 1976.

\bibitem{Stein70book}
Elias~M. Stein.
\newblock {\em Singular integrals and differentiability properties of
  functions}.
\newblock Princeton Mathematical Series, No. 30. Princeton University Press,
  Princeton, 1970.

\bibitem{Stein93book}
Elias~M. Stein.
\newblock {\em Harmonic analysis: real-variable methods, orthogonality, and
  oscillatory integrals}, volume~43 of {\em Princeton Mathematical Series}.
\newblock Princeton University Press, Princeton, NJ, 1993.
\newblock With the assistance of Timothy S. Murphy, Monographs in Harmonic
  Analysis, III.

\bibitem{Sto22}
Dmitriy Stolyarov.
\newblock On {F}ourier multipliers with rapidly oscillating symbols.
\newblock \textit{St. Petersburg Math. J.} Accepted for publication. Available
  at: https://arxiv.org/abs/2203.04881, 2022.

\bibitem{TaoBook}
Terence Tao.
\newblock {\em Nonlinear dispersive equations}, volume 106 of {\em CBMS
  Regional Conference Series in Mathematics}.
\newblock Conference Board of the Mathematical Sciences, Washington, DC; by the
  American Mathematical Society, Providence, RI, 2006.
\newblock Local and global analysis.

\end{thebibliography}
\bibliographystyle{plain}

\end{document}